\documentclass[11pt]{article}
\usepackage{amsfonts}

\usepackage{graphicx}
\usepackage{latexsym}
\usepackage{amsmath}
\usepackage{layout}

\newtheorem{lemma}{Lemma}

\newtheorem{corollary}{Corollary}
\newtheorem{theorem}{Theorem}

\begin{document}
\title{ Estimation of the drift of  fractional Brownian motion }
\author{ Khalifa Es-Sebaiy $^{1, 3}\quad$ Idir Ouassou $^{2}\quad$ Youssef
Ouknine $^{1}\vspace*{0.1in}$ \\
$^{1}$ Department of Mathematics, Faculty of Sciences Semlalia,\\
Cadi Ayyad University 2390 Marrakesh, Morocco. \vspace*{0.1in}\\
$^{2}$ ENSA, Cadi Ayyad University,  Marrakesh, Morocco,\vspace*{0.1in}\\
$^{3}$SAMOS/MATISSE, Centre d'Economie de La Sorbonne,\\
Universit\'e de Panth\'eon-Sorbonne Paris 1,\\
90, rue de Tolbiac, 75634 Paris Cedex 13, France.} \maketitle
\begin{abstract}
We consider the problem of efficient estimation for  the drift of
fractional Brownian motion $B^H:=\left(B^H_t\right)_{ t\in[0,T]}$
with hurst parameter $H$ less than $\frac{1}{2}$. We  also construct
superefficient James-Stein type estimators which dominate, under the
usual quadratic risk, the natural maximum likelihood  estimator.

\end{abstract}




\vskip0.2cm




\section{Introduction}
Fix $H\in(0,1)$ and $T>0$. Let $B^{H}
=\left\{(B^{H,1}_t,...,B^{H,d}_t);t\in[0,T]\right\}$ be a
$d$-dimensional fractional Brownian motion (fBm) defined on the
probability space $\left(\Omega,\mathcal{F},P\right)$. That is,
$B^{H}$ is a zero mean Gaussian vector whose components are
independent one-dimensional fractional Brownian motions with Hurst
parameter $H\in(0,1)$, i.e., for every $i=1,...,d$ $B^{H,i}$ is a
Gaussian process and covariance function given by
\[E(B_s^{H,i}B_t^{H,i})=\frac{1}{2}\left(s^{2H}+t^{2H}-|t-s|^{2H}\right),\quad s,t\in[0,T].\]
For each $i=1,\ldots,d$, $\left(\mathcal{F}_t^i\right)_{t\in[0,T]}$
denotes the filtration generated by
$\left(B^{H,i}_t\right)_{t\in[0,T]}$.
\\The fBm was first introduced by  \cite{kolmogorov} and studied by
 \cite{MV}. Notice that if
$H=\frac{1}{2}$, the process $B^{\frac{1}{2}}$ is a standard
Brownian motion.   However, for $H\neq \frac{1}{2}$, the fBm is
neither a Markov process, nor a semi-martingale.

Let $M$ be a subspace of the Cameron-Martin space defined by
\begin{eqnarray*}M=\left\{\varphi:[0,T]\rightarrow
\mathbb{R}^d; \varphi_t^i=\int_0^t\dot{\varphi}_s^ids \mbox{ with
}\dot{\varphi}^i\in L^2([0,T])  \right.\\
\left.\phantom{\int_0^t}\mbox{ and  }\varphi^i\in
I_{0^+}^{H+\frac{1}{2}}\left(L^2([0,T]) \right), i=1,...,d
\right\}.\end{eqnarray*} Let
$\theta=\left\{(\theta_t^1,\ldots,\theta_t^d);t\in[0,T]\right\}$ be
a function belonging to $M$. Then, Applying Girsanov theorem (see
Theorem 2 in \cite{NO}), there exist a probability measure
absolutely continuous with respect to $P$ under which  the process
$\widetilde{B}^H$ defined by

\begin{eqnarray}\label{girsanovtranform}\widetilde{B}^H_t=B^H_t-\theta_t,\qquad t\in[0,T]\end{eqnarray} is a
fBm with Hurst parameter $H$ and mean zero.  In this case, we say
that, under the probability $P_{\theta}$, the process $B^H$ is a fBm
with drift $\theta$.

 We  consider in this paper the problem of estimating the drift
$\theta$ of  $B^H$ under the probability $P_{\theta}$, with hurst
parameter $H<1/2$. We wish to estimate $\theta$ under the usual
quadratic risk, that is defined for any estimator $\delta$ of
$\theta$ by
$${\cal
R}(\theta,\delta)=E_{\theta}\left[\int_0^T||\delta_t-\theta_t||^2
dt\right]$$ where $E_{\theta}$ is the expectation with respect to a
probability $P_{\theta}$.

 Let $X= (X^1,\ldots , X^d)$ be a normal vector
with  mean $\theta = ({\theta}^1, \ldots , {\theta}^d)\in
~\mathbb{R}^d$ and identity covariance matrix $\sigma^2I_d$.
  The usual maximum likelihood   estimator  of ${\theta}$   is $X$ itself. Moreover, it is efficient in the sense that the
   Cramer-Rao bound over all unbiased  estimators is attained by $X$. That is
$$\sigma^2d=E\left[ \|X-{\theta}\|^2_d\right]=\inf_{\xi\in \mathcal{S}}E\left[ \|\xi-{\theta}\|^2_d\right],  $$
where $\mathcal{S}$ is the class of unbiased estimators of
${\theta}$ and $\|.\|_d$ denotes the Euclidean norm on
$\mathbb{R}^d$.

 \cite{c-s-56} constructed biased superefficient estimators of
${\theta}$ of the form
$$\delta_{a,b}(X) = \left(1 -\frac{ b}{ a +
||X||^2}  \right)X$$  for $a$ sufficiently small and $b$
sufficiently large when $d\geq3$.  \cite{J-S-61} sharpened later
this result and presented an explicit class of biased superefficient
estimators of the form
$$\left(1-\frac{a}{||X||_d^2}\right)X,\  \mbox{for}\ 0 < a < 2(d - 2).$$

Recently, an infinite-dimensional extension of this result  has been
given by  \cite{PR}. The  authors constructed unbiased estimators of
the drift $\left(\theta_t\right)_{t\in [0,T]}$ of a continuous
Gaussian martingale $\left(X_t\right)_{t\in [0,T]}$ with quadratic
variation $\sigma_t^2dt$, where $\sigma\in L^2([0,T],dt)$ is an a.e.
non-vanishing function. More precisely, they proved that
$\hat{\theta}= \left(X_t\right)_{t\in [0,T]}$ is an efficient
estimator of $\left(\theta_t\right)_{t\in [0,T]}$. On the other
hand, using Malliavin calculus, they constructed superefficient
estimators for the drift
  of a Gaussian processe of the form:
$$X_t:=\int_0^tK(t,s)dW_s,\qquad t\in [0,T],$$
where $(W_t)_{t\in [0,T]}$ is a standard Brownian motion and
$K(.,.)$ is a deterministic kernel. These estimators are biased and
of the form $X_t + D_t \log F$,  where $F$ is a positive
superharmonic random variable and $D$ is the Malliavin derivative.

 In Section 3, we prove, using technic
based on the fractional calculus and  Girsanov theorem,  that
$\widehat{\theta}=B^H$ is an
 efficient estimator of $\theta$ under the probability $P_{\theta}$  with risk
\[{\cal
R}(\theta,B^H)=E_\theta\left[\int_0^T\|B_t^H-\theta_t\|^2dt\right]=\frac{T^{2H+1}}{2H+1}d.\]
 Moreover, we will establish  that
$\hat{\theta}=B^H$ is a maximum likelihood estimator of $\theta$.

In Section 4, we construct a class  of biased estimators of
James-Stein type of the form
$$\delta(B^H)_t=\left(1-at^{2H}\left(\frac{r(\|B^H_t\|^2)}{\|B^H_t\|^2}\right)\right)B^H_t,\qquad  t\in[0,T].$$
We give sufficient conditions on the function $r$ and on the
constant $a$ in order that $\delta(B^H)$ dominates $B^H$ under the
usual quadratic risk i.e.
\begin{eqnarray}\label{jemes-stein}{\cal
R}\left(\theta,\delta\left(B^H\right)\right)<{\cal
R}\left(\theta,B^H\right) \qquad  \mbox{ for all } \theta\in M.
\end{eqnarray}

\section{Preliminaries }
This section contains the elements from fractional calculus that we
will need in the paper.
\\
 The fractional Brownian motion $B^H$ has the
following stochastic integral
 representation (see for instance, \cite{AMN}, \cite{N2})

 \begin{eqnarray}\label{repre integral} B^{H,i}_t=\int_0^tK_H(t,s)dW^i_s,\quad
 i=1,...,d;\quad t\in[0,T]
 \end{eqnarray}
 where $W=(W^1,...,W^d)$ denotes the d-dimensional  Brownian
 motion and the kernel $K_H(t,s)$ is equal to
 \begin{eqnarray*}c_H(t-s)^{H-\frac{1}{2}}+c_H(\frac{1}{2}-H)\int_s^t(u-s)^{H-\frac{3}{2}}
 \left(1-(\frac{s}{u})^{\frac{1}{2}-H}\right)du\quad \mbox{ if }
 H\leq\frac{1}{2}\\c_H(H-\frac{1}{2})\int_s^t(u-s)^{H-\frac{3}{2}}
 \left(\frac{s}{u}\right)^{H-\frac{1}{2}}du\quad \mbox{ if }
 H>\frac{1}{2},
 \end{eqnarray*} if $s<t$ and $K_H(t,s)=0$ if $s\geq t$. Here $c_H$ is the normalizing constant
  \[c_H=\sqrt{\frac{2H\Gamma(\frac{3}{2}-H)}{\Gamma(H+\frac{1}{2})\Gamma(2-2H)}}\] where $\Gamma$ is the Euler function.
\\
 We recall some basic definitions and results on classical fractional calculus which we will
 need. General information about
fractional calculus can be found in \cite{SKM}.\\ The left
fractional Riemann-Liouville integral of $f\in L^1((a,b))$ of order
$\alpha>0$ on $(a, b)$ is given at almost all $x\in(a,b)$ by
\[I_{a^+}^{\alpha}f(x)=\frac{1}{\Gamma(\alpha)}\int_a^x(x-y)^{\alpha-1}f(y)dy.\]
If $f\in I_{a^+}^{\alpha}(L^p(a,b))$ with $0<\alpha<1$ and $p>1$
then the left-sided Riemann-Liouville derivative of $f$ of order
$\alpha$ defined
by\[D_{a^+}^{\alpha}f(x)=\frac{1}{\Gamma(1-\alpha)}\left(\frac{f(x)}{(x-a)^\alpha}+\alpha\int_a^x\frac{f(x)-f(y)}{(x-y)^{\alpha+1}}dy\right)\]
for almost all $x\in(a, b).$ \\For $H\in(0,1)$, the integral
transform
\[(K_Hf)(t)=\int_0^tK_H(t,s)f(s)ds\]
is a isomorphism from $L^2([0,1])$ onto
$I_{0^+}^{H+\frac{1}{2}}\left(L^2([0,1])\right)$ and its inverse
operator $K_H^{-1}$ is given by
\begin{eqnarray}K_H^{-1}f&=&t^{H-\frac{1}{2}}D_{0^+}^{H-\frac{1}{2}}t^{\frac{1}{2}-H}f'\quad\mbox{ for } H>1/2,\\
K_H^{-1}f&=&t^{\frac{1}{2}-H}D_{0^+}^{\frac{1}{2}-H}t^{H-\frac{1}{2}}D_{0^+}^{2H}f\quad\mbox{
for }H<1/2.
\end{eqnarray}
Moreover, for $H<\frac{1}{2}$, if $f$ is an  absolutely continuous
function then $K_H^{-1}f$ can be represented  of the form ( see
\cite{NO} )\begin{eqnarray}\label{inverse of
K_h}K_H^{-1}f&=&t^{H-\frac{1}{2}}I_{0^+}^{\frac{1}{2}-H}t^{\frac{1}{2}-H}f'.
\end{eqnarray}

\section{The maximum likelihood estimator and Cramer-Rao type bound}
  We consider a  function
$\theta=\left(\theta^1,\ldots,\theta^d\right)$ belonging to $M$. An
estimator   $\xi=(\xi^1,\ldots,\xi^d)$  of
$\theta=(\theta^1,\ldots,\theta^d)$ is called unbiased if, for every
$t\in[0,T]$
 $${E}_\theta(\xi_t^{i})=\theta_t^{i},\qquad i=1,\ldots,d$$
 and it is called adapted if, for each $i=1,\ldots,d$, $\xi^i$ is
 adapted to $\left(\mathcal{F}_t^i\right)_{t\in[0,T]}$.
\\
 Since for any $i=1,...,d$, the function $\theta^i$ is
deterministic  and
\[\int_0^T(K_H^{-1}(\theta^i)(s))^2ds<\infty,\] then
Girsanov theorem yields that there exists a probability measure
${P_{\theta}}$  absolutely continuous with respect to $P$ under
which the process $\widetilde{B}^H:=\left(\widetilde{B}^H_t;
t\in[0,T]\right)$ defined by

\begin{eqnarray}\label{girsanovtranform}\widetilde{B}^H_t=B^H_t-\theta_t,\qquad t\in[0,T]\end{eqnarray} is a
d-dimensional fBm with Hurst parameter $H$ and mean zero. Moreover
the Girsanov density   of ${P_{\theta}}$ with respect to $P$ is
given by:

 \[\frac{d{P_{\theta}}}{dP}=
 \exp\left[\sum_{i=1}^{d}\left(\int_0^TK_H^{-1}(\theta^i)(s)dW^i_s-\frac{1}{2}\int_0^T(K_H^{-1}(\theta^i)(s))^2ds\right)\right]\]
 and
 \[\widetilde{B}^H_t=\int_0^tK_H(t,s)d\widetilde{W}_s\] where $\widetilde{W}$ is a d-dimensional Brownian motion under the probability
 ${P_{\theta}}$
 and  \[\widetilde{W}_t^i=W_t^i-\int_0^tK_H^{-1}(\theta^i)(s)ds,\quad i=1,...,d; \quad t\in[0,T].\]
 The equation (\ref{girsanovtranform}) implies that $B^H$  is an unbiased and
 adapted estimator  of $\theta$ under probability $P_{\theta}$.
 In addition, we obtain the Cramer-Rao type bound:
 \[R(H,\hat{\theta}):={\cal
R}(\theta,B^H)=\int_0^T{E}_\theta\|\widetilde{B}_t^H\|^2dt=
 d\int_0^Tt^{2H}dt=\frac{T^{2H+1}}{2H+1}d.\]

 The first main  result of this section is given by  the following proposition which asserts that  $\widehat{\theta}=B^H$ is an
efficient estimator of $\theta$.
\begin{theorem}Assume that $H<\frac{1}{2}$. If $\xi$ is an unbiased and adapted estimator of
$\theta$, then
\begin{eqnarray}\label{cramer rao}{E}_\theta\int_0^T\|\xi_t-\theta_t\|^2dt\geq
R(H,\hat{\theta}).
\end{eqnarray}
\end{theorem}
\begin{proof}  Since $\xi$  is unbiased, then for every $\varphi\in
M$ we have
\[{E}_{\varphi}(\xi_t^{j})={E}_{\varphi}(\varphi_t^{j}),\qquad j=1,\ldots,d.\]
Let $\varphi=\theta+\varepsilon\psi$ with $\psi \in M$ and
$\varepsilon\in\mathbb{R}$. Then for every $t\in[0,T]$  and
$j\in\{1,\ldots,d\}$, we have
\begin{eqnarray*}{E}_{\theta+\varepsilon\psi}(\xi_t^{j})&=&{E}_{\theta+\varepsilon\psi}(\theta_t^{j}+\varepsilon\psi_t^{j})\\
&=&{E}_{\theta+\varepsilon\psi}(\theta_t^{j})+\varepsilon\psi_t^{j}.
\end{eqnarray*}
This implies that for every $j=1,\ldots,d$
\begin{eqnarray*}\psi_t^j&=&\frac{d}{d\varepsilon}_{/\varepsilon=0}{E}_{\theta+\varepsilon\psi}(\xi_t^{j}-\theta_t^{j})\\
&=&E\left(\frac{d}{d\varepsilon}_{/\varepsilon=0}
\exp\left[\sum_{i=1}^{d}\left(\int_0^tK_H^{-1}(\theta^i+\varepsilon\psi^i)(s)dW_s^i\right.\right.
\right.\\&&\left.\left.\left.\qquad\qquad\qquad\qquad\quad-\frac{1}{2}\int_0^t(K_H^{-1}(\theta^i+\varepsilon\psi^i)(s))^2
)
ds\right)\right](\xi_t^j-\theta_t^j)\right)\\
&=&{E}_{\theta}\left(\sum_{i=1}^{d}\left[\int_0^tK_H^{-1}(\psi^i)(s)dW_s^i-\int_0^tK_H^{-1}(\psi^i)(s)K_H^{-1}
(\theta^i)(s)ds\right]\right.\\&&\qquad\qquad\qquad\qquad\qquad\qquad\qquad\qquad\qquad\left.\phantom{\int_.^.}\times(\xi_t^{j}-\theta_t^{j})\right)\\&=&
{E}_{\theta}\left(\sum_{i=1}^{d}\left[\int_0^tK_H^{-1}(\psi^i)(s)d\widetilde{W}_s^i\right](\xi_t^j-\theta_t^{j})\right)\\&=&
{E}_{\theta}\left(\left[\int_0^tK_H^{-1}(\psi^j)(s)d\widetilde{W}_s^j\right](\xi_t^j-\theta_t^{j})\right).
\end{eqnarray*}
Applying Cauchy-Schwarz inequality in $L^2(\Omega,d{P}_{\theta})$,
we obtain that for every $t\in[0,T]$
\begin{eqnarray*}\|\psi_t\|^2=\sum_{j=1}^{d}(\psi_t^j)^2&\leq&\sum_{j=1}^{d}{E}_{\theta}\left((\xi_t^{j}-\theta_t^{j})^2\right)
{E}_{\theta}\left(\left[\int_0^tK_H^{-1}(\psi^{j})(s)d\widetilde{W}_s^{j}\right]^2\right)\\
&&=\sum_{j=1}^{d}{E}_{\theta}\left[\left((\xi_t^{j}-\theta_t^{j})^2\right)\int_0^t(K_H^{-1}(\psi^{j})(s))^2ds\right].\end{eqnarray*}
We take for each $j=1,\ldots,d$, $\psi_t^{j}={t^{2H}}$.  Since
$t\longrightarrow t^{2H}$ is absolutely continuous function, then by
(\ref{inverse of K_h}), a simple calculation shows that
\begin{eqnarray*}K_H^{-1}(t^{2H})&=&2Ht^{H-\frac{1}{2}}I_{0^+}^{\frac{1}{2}-H}t^{H-\frac{1}{2}}
\\&=&\frac{2H\beta(\frac{1}{2}-H,H+\frac{1}{2})}{\Gamma(\frac{1}{2}-H)}t^{H-1/2}\\&=&2H(\Gamma(\frac{1}{2}+H))t^{H-1/2}.\end{eqnarray*}
It is known that\begin{eqnarray}\label{gamma}
0\leq\Gamma(z)\leq1\quad\mbox{for every } z\in[1,2].\end{eqnarray}
Combining the facts that $z\Gamma(z)=\Gamma(z+1),$ $z>0$,
$2H\leq(H+\frac{1}{2})^2$ and (\ref{gamma}), we obtain
\begin{eqnarray*} dt^{2H}=\|\psi_t\|^2&\leq&
(\Gamma(\frac{3}{2}+H))^2{E}_{\theta}\left(\|\xi_t-\theta_t\|^2\right)\leq{E}_{\theta}\left(\|\xi_t-\theta_t\|^2\right).\end{eqnarray*}
 Hence, by an integration with respect to $dt$, we get
\[R(H,\hat{\theta})=\frac{T^{2H+1}}{2H+1}\leq{E}_{\theta}\int_0^T\|\xi_t-\theta_t\|^2dt.\]
Therefore (\ref{cramer rao}) is satisfied.
\end{proof}
\begin{corollary}The process $\hat{\theta}=B^H$ is a maximum
likelihood estimator of $\theta$.
\end{corollary}
\begin{proof} We have for every $\psi\in
M$\[\frac{d}{d\varepsilon}_{/\varepsilon=0}\exp\left[\sum_{i=1}^{d}\int_0^tK_H^{-1}(\hat{\theta}^i+
\varepsilon\psi^i)(s)dW_s^i-\frac{1}{2}\int_0^t(K_H^{-1}(\hat{\theta}^i+\varepsilon\psi^i)(s))^2
) ds\right]=0.\] Hence
\[\sum_{i=1}^{d}\left(\int_0^tK_H^{-1}(\psi^i)(s)dW_s^i-\int_0^tK_H^{-1}(\psi^i)(s)K_H^{-1}(\hat{\theta^i})(s)ds\right)=0.\]
Which implies that for every $i=1,...,d$
\[E\left(\int_0^tK_H^{-1}(\psi^i)(s)dW_s^i-\int_0^tK_H^{-1}(\psi^i)(s)K_H^{-1}(\hat{\theta^i})(s)ds\right)^2=0.\]
Then, for each  $i=1,...,d$
\[W_t^i=\int_0^tK_H^{-1}(\hat{\theta^i})(s)ds, \qquad t\in[0,T].\]
Therefore by (\ref{repre integral}), we obtain that
$B^H=\hat{\theta}.$
\end{proof}

\section{Superefficient  James-Stein type estimators}
The aim of this section is to construct superefficient estimators of
$\theta$ which dominate, under the usual quadratic risk, the natural
MLE
 estimator $B^H$. The class of estimators considered here are of the form
\begin{eqnarray}
\label{B+g(B,t)} \delta({B}^H)_t={B}^H_t+g({B}^H_t,t),\qquad
t\in[0,T]
\end{eqnarray}where
$g:\mathbb{R}^{d+1}\longrightarrow\mathbb{R}^d$ is a  function. The
problem turns to find  sufficient conditions on $g$ such that
${\cal{R}}\left(\theta,\delta({B}^H)\right)<\infty$ and the risk
difference  is negative, i.e.
 \begin{eqnarray*}
 \Delta{\cal{R}}(\theta) =
{\cal{R}}\left(\theta,\delta({B}^H)\right)-{\cal{R}}\left(\theta,{B}^H\right)<0.
\end{eqnarray*}
In the sequel we assume that the function $g$ satisfies the
following assumption:$$ (A)\left\{ \begin{array}{c}
              E_{\theta}\left[\int_0^T||g({B}^H_t,t)||_d^2\
dt\right]<\infty,
\phantom{E_{\theta}\left[\int_0^T||g({B}^H_t,t)||_d^2\
dt\right]<\infty}  \\
                \mbox{  the partial derivatives }
\partial_ig^i:=\frac{\partial g^i}{\partial x^i},\ i=1,\ldots,n
\mbox{ of $g$ exist. }
              \end{array}
\right.$$   Then
${\cal{R}}\left(\theta,\delta({B}^H)\right)<\infty$. Moreover
\begin{eqnarray*}
 \Delta{\cal{R}}(\theta)&=&
{E}_{\theta}\left[\int_0^T||{B}^H_t+g({B}^H_t,t)-\theta_t||_d^2-||{B}^H_t-\theta_t||_d^2dt\right]\\
&=&
{E}_{\theta}\left[\int_0^T||g({B}^H_t,t)||_d^2+2\sum_{i=1}^d\left(g^i({B}^H_t,t)({B}^{H,i}_t-\theta^i_t)\right)dt\right].
\end{eqnarray*}
In addition,
\begin{eqnarray*}
&&{E}_{\theta}\int_0^T\sum_{i=1}^d\left(g^i(B^H_t,t)({B}^{H,i}_t-\theta^i_t)\right)dt
\\&=&\sum_{i=1}^d\int_0^T (2\pi t^{2H})^{-\frac{d}{2}}\left(\int_{\mathbb{R}^d}g^i(x^1,\ldots,
x^d,
t)(x^i-\theta_t^i)\right.\\&&\qquad\qquad\qquad\qquad\qquad\left.\phantom{\int_.^.}\times{e^{-\frac{\sum_{j=1}^d(x^j-\theta_t^j)2}{2t^{2H}}}}
dx^1\ldots dx^d\right) dt\\&=&\sum_{i=1}^d\int_0^T(2\pi
t^{2H})^{-\frac{d}{2}}\left(\int_{\mathbb{R}^d}t^{2H}\partial_ig^i(x^1,\ldots,
x^d,
t)~~\right.~~\\~~&&~~\qquad\qquad\qquad\qquad\qquad\left.\phantom{\int_.^.}\times{e^{-\frac{\sum_{j=1}^d(x^j-\theta_t^j)2}{2t^{2H}}}}dx^1\ldots
dx^d\right) dt\\&=&\sum_{i=1}^d\int_0^T\left(t^{2H}{E}_{\theta}
\partial_ig^i(B_t^H,t)\right)dt=
{E}_{\theta}\left[
\sum_{i=1}^d\int_0^Tt^{2H}\partial_ig^i(B_t^H,t)dt
 \right].
\end{eqnarray*}
Consequently, the risk difference equals
\begin{eqnarray}\label{risk}  \Delta{\cal{R}}(\theta)=
{E}_{\theta}\left[\int_0^T\left(||g(B_t^H,t)||^2+2t^{2H}\sum_{i=1}^d\partial_ig^i(B_t^H,t)\right)dt\right].
\end{eqnarray}
We can now state the following theorem.
\begin{theorem}
\label{t0} Let $g:\mathbb{R}^{d+1}\longrightarrow\mathbb{R}^d$ be a
 function satisfying (A). A sufficient conditions for the
estimator $\left(B^H_t+g(B^H_t,t)\right)_{t\in[0,T]}$ to dominate
$B^H$ under the usual quadratic risk is
$${E}_{\theta}\left[\int_0^T\left(||g(B_t^H,t)||^2+2t^{2H}\sum_{i=1}^d\partial_ig^i(B_t^H,t)\right)dt\right]<0.$$
\end{theorem}

 As an  application, take $g$ of the form
\begin{eqnarray}\label{james}
g(x,t)=at^{2H}\frac{r\left(\|x\|^2\right)}{\|x\|^2}x,
\end{eqnarray}where $a$ is a constant strictly positive and  $r:\mathbb{R}^+\rightarrow\mathbb{R}^+$ is bounded derivable
function.\\ The next lemma give a sufficient condition for  $g$ in
(\ref{james}) to satisfies the assumption $(A)$.
\begin{lemma}If $d\geq3$ and $H<\frac{1}{2}$ then
\begin{eqnarray}\label{finiteintegral}E\left[\int_0^T\frac{1}{\|B_t^H\|^2}dt\right]<\infty.
\end{eqnarray}
\end{lemma}
\begin{proof}Firstly the integral given by (\ref{finiteintegral})
 is well defined,  because \[(dt\times P)\left((t,w);B^H_t(w)=0\right)=0\]
where $(dt\times P)$ is the product measure.
 \\Using
 the change of variable and $d\geq3$ we see that
\begin{eqnarray*}E\int_0^T\frac{1}{\|B_t^H\|^2}dt=
\int_0^T\frac{dt}{t^{2H}}\int_{\mathbb{R}^d}\frac{e^{-\frac{\|y\|^2}{2}}}{\sqrt{2\pi}\|y\|^2}dy\leq
C\int_0^T\frac{1}{t^{2H}}dt,
\end{eqnarray*} where $C$ is a constant depending only on $d$.
Furthermore, since $H<\frac{1}{2}$ then (\ref{finiteintegral})
holds.
\end{proof}

\begin{theorem}
\label{t1} Assume that  $d\geq3$. If the function $r$, the constant
$a$ and the parameter $H$ satisfy:
\begin{enumerate}
\item[i)] $0\leq r(.)\leq1$
\item[ii)]$r(\cdot)$ is  differentiable
and increasing
\item[iii)] $0<a\leq2(d-2)$ and $ H<1/2$,
\end{enumerate}
then the estimator
$$\delta(B^H)=B^H_t-at^{2H}\frac{r\left(\|B_t^H\|^2\right)}{\|B_t^H\|^2}B_t^H,\qquad
t\in[0,T].$$ dominates $B^H$.
\end{theorem}
\begin{proof} It suffices to prove that $\Delta
{\cal{R}}(\theta)<0$. From (\ref{risk}) and the hypothesis  $i)$ and
$ii)$ we can write
\begin{eqnarray*}
\Delta
{\cal{R}}(\theta)&=&a{E}_{\theta}\left[\int_0^Tt^{4H}\left(\frac{ar^2(\|B^H_t\|^2)}{\|B^H_t\|^2}-2(d-2)\frac{r(\|B^H_t\|^2)}{\|B^H_t\|^2}
\right.\right.\\&&\qquad\qquad\qquad\qquad\qquad\qquad\left.\left.\phantom{\int_.^.}-4r'(\|B^H_t\|^2)\right)dt\right]
\\
&\leq&
a\left[a-2(d-2)\right]{E}_{\theta}\left[a\int_0^Tt^{4H}\frac{r(\|B^H_t\|^2)}{\|B^H_t\|^2}\right].
\end{eqnarray*}
Combining this fact with the assumption  $iii)$ yields  that the
risk difference is negative. Which proves the desired result.
\end{proof}\\
For $r=1$, we obtain  a James-Stein type estimator:
\begin{corollary} Let $d\geq3$,
$0<H<\frac{1}{2}$ and $0<a\leq2(d-2)$. Then the  estimator
$$\left(1-\frac{at^{2H}}{\|B_t^H\|^2}\right)B_t^H,\qquad{t\in[0,T]}$$
dominates  $B^H$.
\end{corollary}

\vspace{1cm} \noindent {\bf{Acknowledgement} }\\ The authors would
like to thank the editor Hira Koul and referees for several helpful
corrections and suggestions that led to many improvements in the
paper.




\end{document}